\documentclass[reqno,11pt]{article} %leqno is the option to put formula numbers on the left side
\usepackage{amsmath, amssymb}
\usepackage{amsthm} %theorem environment option
\theoremstyle{plain} %text of this environment is typesetted in italics
\newtheorem{theorem}{\indent\sc Theorem}[section]
\newtheorem{lemma}[theorem]{\indent\sc Lemma}
\newtheorem{corollary}[theorem]{\indent\sc Corollary}

\theoremstyle{definition} %text of this environment is typesetted in roman letters

\newtheorem{remark}[theorem]{\indent\sc Remark}
\newtheorem{example}[theorem]{\indent\sc Example}

\makeatletter
\def\address#1#2{\begingroup
\noindent\parbox[t]{7.8cm}{%
\small{\scshape\ignorespaces#1}\par\vskip1ex
\noindent\small{\itshape E-mail address}%
\/: #2\par\vskip4ex}\hfill%
\endgroup}%
\makeatother
\title{Some applications of eta-quotients} %title of the paper
\author{
\textsc{Ick Sun Eum, Ja Kyung Koo and Dong Hwa Shin$^*$} %names of authors
}
\date{} %leave empty
\begin{document}

\maketitle

%%%%%%%%%%%%%%% footnote %%%%%%%%%%%%%%%%
\footnote{ %2010 MSC numbers
2010 \textit{Mathematics Subject Classification}. Primary 11F20, Secondary 11F11, 11R37.}
\footnote{ %key words and phrases
\textit{Key words and phrases}. Eta-quotients, modular forms, ring class fields.}
\footnote{
\thanks{
The second named author was partially supported by the NRF of Korea grant
funded by the MISP (2013042157).
$^*$The corresponding author was
supported by Hankuk University of Foreign Studies Research Fund of
2013.} }
%%%%%%%%%%%%%%%%%%%%%%%%%%%%%%%%%%%%%%%

\begin{abstract}
We show that every modular form  on $\Gamma_0(2^n)$ ($n\geq2$)
can be expressed as a sum of eta-quotients.
Furthermore, we construct a primitive generator of the ring class field
of the order of conductor $4N$ ($N\geq1$) in an imaginary quadratic field
in view of the special value of certain eta-quotient.
\end{abstract}

\section {Introduction}

For a positive integer $N$ let
\begin{equation*}
\Gamma_0(N)=\{\left[\begin{matrix}a&b\\c&d\end{matrix}\right]\in\mathrm{SL}_2(\mathbb{Z})~|~
c\equiv0\pmod{N}\}.
\end{equation*}
This group acts on the complex upper-half plane $\mathbb{H}$
as fractional linear transformations, and gives rise to
the modular curve $X_0(N)=\Gamma_0(N)\backslash\mathbb{H}^*$, where
$\mathbb{H}^*=\mathbb{H}\cup\mathbb{Q}\cup\{i\infty\}$ \cite[Chapter 1]{Shimura}.
We denote its (meromorphic) function field by $\mathbb{C}(X_0(N))$. Then it is well-known that
$\mathbb{C}(X_0(N))=\mathbb{C}(j(\tau),j(N\tau))$, where
$j(\tau)$ is the elliptic modular function whose Fourier expansion
with respect to $q=e^{2\pi\tau}$
has integer Fourier coefficients as follows:
\begin{equation*}
j(\tau)=1/q+744+196884q+21493760q^2+864299970q^3+\cdots\quad(\tau\in\mathbb{H})
\end{equation*}
\cite[$\S$4.1 and $\S$6.4, Theorem 7]{Lang}.
\par
We define the \textit{Dedekind eta-function} $\eta(\tau)$ by
\begin{equation}\label{eta}
\eta(\tau)=q^{1/24}\prod_{n=1}^\infty(1-q^n)\quad(\tau\in\mathbb{H}),
\end{equation}
and call a product of the form
$\prod_{d|N}\eta(d\tau)^{m_d}$ with $m_d\in\mathbb{Z}$
an \textit{eta-quotient}. Ono suggested the following question \cite[Problem 1.68]{Ono}:
Classify the spaces of modular forms which are generated by eta-quotients.
And, Choi \cite{Choi} recently showed that if $X_0(N)$ is of
genus zero (that is, in the cases $N=1,\ldots,10,12,13,16,18,25$), then every modular
form on $\Gamma_0(N)$ can be expressed as a $\mathbb{C}$-linear combination of certain eta-quotients.
\par
Denote by $\mathcal{R}_0(N)$ the integral closure of the polynomial ring $\mathbb{C}[j(\tau)]$ in the function field $\mathbb{C}(X_0(N))$.
In this paper, we shall first construct explicit
generators of the ring $\mathcal{R}_0(2^n)$ ($n\geq2$) over $\mathbb{C}$
in terms of eta-quotients (Theorem \ref{generators}).
As its corollary we can give an answer to Ono's question when $N=2^n$
(Corollary \ref{sum}).
\par
On the other hand, let $K$ be an imaginary quadratic field of discriminant $d_K$, and set
\begin{equation}\label{tauK}
\tau_K=\left\{\begin{array}{ll}
(-1+\sqrt{d_K})/2 & \textrm{if}~d_K\equiv1\pmod{4},\\
\sqrt{d_K}/2 & \textrm{if}~d_K\equiv0\pmod{4}.
\end{array}\right.
\end{equation}
Let $H_{K,N}$ be the ring class field of the order of conductor $N$ in $K$.
As a consequence of the main theorem of the theory of complex multiplication,
we have
\begin{equation}\label{j(NtauK)}
H_{K,N}=K(j(N\tau_K))
\end{equation}
\cite[$\S$10.3, Theorem 5]{Lang}.
We shall show that if $N\equiv0\pmod{4}$, then
$256\eta(N\tau_K)^8/\eta((N/4)\tau_K)^8$ becomes a primitive generator
of $H_{K,N}$ over $K$ (Theorem \ref{ringclassinvariant}).
\par
For a number field $F$, we denote by $\mathcal{O}_F$ its ring of integers.
We shall further verify the structure of
$\mathcal{O}_{H_{K,2^n}}[1/2]$ ($n\geq3$)
over $\mathcal{O}_{H_{K,4}}[1/2]$  in view of the special values of eta-quotients
generating $\mathcal{R}_0(2^n)$ (Theorem \ref{2integers}). To this end, we shall introduce
an explicit version of Shimura's reciprocity law.

\section {Modular forms and functions}

We shall briefly examine the modularity of eta-quotients, Weierstrass functions
and Siegel functions.

\begin{lemma}\label{modularity}
Let $N$ be a positive integer.
Assume that a family of integers $\{m_d\}_{d|N}$, where $d$ runs over all positive divisors of $N$, satisfies the following conditions:
\begin{itemize}
\item[\textup{(i)}] $\sum_{d|N}m_d$ is even.
\item[\textup{(ii)}] $\sum_{d|N}dm_d\equiv\sum_{d|N}(N/d)m_d\equiv0\pmod{24}$.
\item[\textup{(iii)}] $\prod_{d|N}d^{m_d}$
is a square in $\mathbb{Q}$.
\end{itemize}
Then, the eta-quotient $\prod_{d|N}\eta(d\tau)^{m_d}$
is a meromorphic modular form of
weight $(1/2)\sum_{d|N}m_d$
on $\Gamma_0(N)$ having rational Fourier coefficients with respect to $q$.
\end{lemma}
\begin{proof}
See \cite[Theorem 1.64]{Ono} and the definition (\ref{eta}).
\end{proof}

\begin{remark}\label{no}
Every eta-quotient has neither zeros nor poles on $\mathbb{H}$ by the definition (\ref{eta}).
\end{remark}

For a lattice $\Lambda$ in $\mathbb{C}$, the \textit{Weierstrass $\wp$-function} relative to $\Lambda$ is defined by
\begin{equation*}
\wp(z;\Lambda)=\frac{1}{z^2}+\sum_{\omega\in\Lambda\setminus\{0\}}\bigg(\frac{1}{(z-\omega)^2}
-\frac{1}{\omega^2}\bigg)\quad(z\in\mathbb{C}).
\end{equation*}
This is a meromorphic function on $z$ and is periodic with respect to $\Lambda$.
Moreover, for $z_1,z_2\in\mathbb{C}\setminus\Lambda$ we have the assertion
\begin{equation}\label{pequal}
\wp(z_1;\Lambda)=\wp(z_2;\Lambda)
~\Longleftrightarrow~ z_1\equiv\pm z_2\pmod{\Lambda}
\end{equation}
\cite[Chapter IV, $\S$3]{Silverman}.
\par
Let $N\geq2$ and $\mathbf{v}=\left[\begin{matrix}v_1\\v_2\end{matrix}\right]\in
(1/N)\mathbb{Z}^2\setminus\mathbb{Z}^2$. We define
\begin{equation*}
\wp_\mathbf{v}(\tau)=\wp(v_1\tau+v_2;[\tau,1])\quad(\tau\in\mathbb{H}).
\end{equation*}
It is a meromorphic modular
form of weight $2$ on the principal congruence subgroup $\Gamma(N)=\{\gamma\in\mathrm{SL}_2(\mathbb{Z})~|~
\gamma\equiv\pm I_2\pmod{N}\}$. Furthermore, it depends only on $\pm\mathbf{v}\pmod{\mathbb{Z}^2}$ and satisfies the following transformation formula:
If $\alpha=\left[\begin{matrix}a&b\\c&d\end{matrix}\right]\in\mathrm{SL}_2(\mathbb{Z})$, then \begin{equation}\label{transformation}
\wp_\mathbf{v}(\tau)\circ\alpha=
(c\tau+d)^2\wp_{\alpha^T\mathbf{v}}(\tau),
\end{equation}
where $\alpha^T$ stands for the transpose of $\alpha$ \cite[$\S$6.2]{Lang}.

\begin{lemma}\label{plevel}
Let $n\geq2$ and $\mathbf{v}=\left[\begin{matrix}v_1\\v_2\end{matrix}\right]\in(1/2)\mathbb{Z}^2
\setminus\mathbb{Z}^2$.
\begin{itemize}
\item[\textup{(i)}]
If $\alpha=\left[\begin{matrix}a&b\\c&d\end{matrix}\right]\in
\Gamma_0(2^{n-1})$, then we have
\begin{equation*}
\wp_\mathbf{v}(2^{n-1}\tau)\circ\alpha=(c\tau+d)^2\wp_{\left[\begin{smallmatrix}v_1+(c/2^{n-1})v_2\\
v_2\end{smallmatrix}\right]}(2^{n-1}\tau).
\end{equation*}
\item[\textup{(ii)}] $\wp_{\mathbf{v}}(2^{n-1}\tau)$ is a
meromorphic modular form of weight $2$ on $\Gamma_0(2^n)$.
\end{itemize}
\end{lemma}
\begin{proof}
(i) Note that $a,d\equiv1\pmod{2}$. We derive that
\begin{eqnarray*}
\wp_\mathbf{v}(2^{n-1}\tau)\circ\alpha&=&
\wp_\mathbf{v}(2^{n-1}(a\tau+b)/(c\tau+d))\\
&=&\wp_\mathbf{v}(\tau)
\circ\left[\begin{matrix}2^{n-1}a&2^{n-1}b\\c&d\end{matrix}\right]\\
&=&\wp_\mathbf{v}(\tau)\circ\left[\begin{matrix}a & 2^{n-1}b\\c/2^{n-1}&d\end{matrix}\right]
\circ\left[\begin{matrix}2^{n-1}&0\\0&1\end{matrix}\right]\\
&=&(((c/2^{n-1})\tau+d)^2\wp_{\left[\begin{smallmatrix}a & 2^{n-1}b\\c/2^{n-1}&d\end{smallmatrix}\right]^T\mathbf{v}}(\tau))\circ\left[\begin{matrix}2^{n-1}&0\\0&1\end{matrix}\right]
\quad\textrm{by (\ref{transformation})}\\
&=&((c/2^{n-1})(2^{n-1}\tau)+d)^2\wp_{\left[\begin{smallmatrix}av_1+(c/2^{n-1})v_2\\
2^{n-1}bv_1+dv_2\end{smallmatrix}\right]}(2^{n-1}\tau)\\
&=&(c\tau+d)^2\wp_{\left[\begin{smallmatrix}v_1+(c/2^{n-1})v_2\\
v_2\end{smallmatrix}\right]}(2^{n-1}\tau)\\
&&\textrm{because}~
\left[\begin{smallmatrix}av_1+(c/2^{n-1})v_2\\
2^{n-1}bv_1+dv_2\end{smallmatrix}\right]
\equiv
\left[\begin{smallmatrix}v_1+(c/2^{n-1})v_2\\
v_2\end{smallmatrix}\right]\pmod{\mathbb{Z}^2}.
\end{eqnarray*}
(ii) Moreover, let $\left[\begin{matrix}a&b\\c&d\end{matrix}\right]\in\Gamma_0(2^n)$. Since
$c\equiv0\pmod{2^n}$, we get $(c/2^{n-1})v_2\in\mathbb{Z}$.
So we obtain by (i) that
\begin{equation*}
\wp_\mathbf{v}(2^{n-1}\tau)\circ\alpha=(c\tau+d)^2\wp_\mathbf{v}(2^{n-1}\tau).
\end{equation*}
This shows that $\wp_\mathbf{v}(2^{n-1}\tau)$ is a meromorphic modular form of weight $2$ on $\Gamma_0(2^n)$.
\end{proof}

For a positive integer $N$ let $\mathbb{C}(X(N))$ be the function field of the modular curve $X(N)=\Gamma(N)\backslash\mathbb{H}^*$. Then, $\mathbb{C}(X(N))$ is a Galois extension of
$\mathbb{C}(X(1))=\mathbb{C}(X_0(1))$ whose Galois group is naturally isomorphic to
$\overline{\Gamma}(1)/\overline{\Gamma}(N)$, where $\overline{\Gamma}(N)=
\langle\Gamma(N),\pm I_2\rangle/\{\pm I_2\}$ \cite[p.31]{Shimura}.
\par
Let $N\geq2$ and $\mathbf{v}=\left[\begin{matrix}v_1\\v_2\end{matrix}\right]\in(1/N)\mathbb{Z}^2\setminus\mathbb{Z}^2$. We define the \textit{Siegel function} $g_\mathbf{v}(\tau)$ on $\mathbb{H}$ by
\begin{equation}\label{Siegel}
g_\mathbf{v}(\tau)=
-q^{(1/2)(v_1^2-v_1+1/6)}e^{\pi
iv_2(v_1-1)}(1-q^{v_1}e^{2\pi iv_2}) \prod_{n=1}^{\infty}(1-q^{n+v_1}e^{2\pi iv_2})(1-q^{n-v_1}e^{-2\pi iv_2}).
\end{equation}
Then $g_\mathbf{v}(\tau)$ and $g_\mathbf{v}(\tau)^{12N}$ belong to
$\mathbb{C}(X(12N^2))$ and
$\mathbb{C}(X(N))$, respectively \cite[Chapter 3, Theorems 5.2 and 5.3]{K-L}.
Furthermore, if $\mathbf{s}=\left[\begin{matrix}s_1\\s_2\end{matrix}\right]\in\mathbb{Z}^2$,
then $g_\mathbf{v}(\tau)$ satisfies the translation formula
\begin{equation}\label{translation}
g_{\mathbf{v}+\mathbf{s}}(\tau) =(-1)^{s_1s_2+s_1+s_2}e^{-\pi
i(s_1v_2-s_2v_1)}g_\mathbf{v}(\tau)
\end{equation}
\cite[pp.28--29]{K-L}.

\begin{lemma}\label{gtoh}
We have the following relations.
\begin{itemize}
\item[\textup{(i)}]
If $\mathbf{u},\mathbf{v}\in\mathbb{Q}^2\setminus\mathbb{Z}^2$
such that $\mathbf{u}\not\equiv\pm\mathbf{v}\pmod{\mathbb{Z}^2}$, then
we have
\begin{equation*}
\wp_\mathbf{u}(\tau)-\wp_\mathbf{v}(\tau)
=-g_{\mathbf{u}+\mathbf{v}}(\tau)g_{\mathbf{u}-\mathbf{v}}(\tau)\eta(\tau)^4/
g_\mathbf{u}(\tau)^2g_\mathbf{v}(\tau)^2.
\end{equation*}
\item[\textup{(ii)}]
$g_{\left[\begin{smallmatrix}1/2\\0\end{smallmatrix}\right]}
(\tau)g_{\left[\begin{smallmatrix}1/2\\1/2\end{smallmatrix}\right]}
(\tau)g_{\left[\begin{smallmatrix}0\\1/2\end{smallmatrix}\right]}(\tau)=
2e^{\pi i/4}$.
\item[\textup{(iii)}] $g_{\left[\begin{smallmatrix}1/2\\0\end{smallmatrix}\right]}(\tau)=
    -\eta(\tau/2)^2/\eta(\tau)^2$.
\end{itemize}
\end{lemma}
\begin{proof}
(i) See \cite[p.51]{K-L}.\\
(ii) See \cite[Lemma 2.6(ii)]{K-S-Y}.\\
(iii) We see that
\begin{eqnarray*}
g_{\left[\begin{smallmatrix}1/2\\0\end{smallmatrix}\right]}(\tau)&=&
-q^{-1/24}(1-q^{1/2})\prod_{n=1}^\infty(1-q^{n+1/2})(1-q^{n-1/2})\quad\textrm{by the definition (\ref{Siegel})}\\
&=&-q^{-1/24}\prod_{n=1}^\infty(1-q^{(1/2)(2n-1)})^2\\
&=&-q^{-1/24}\prod_{n=1}^\infty(1-q^{(1/2)n})^2/\prod_{n=1}^\infty(1-q^{(1/2)(2n)})^2\\
&=&-\eta(\tau/2)^2/\eta(\tau)^2\quad\textrm{by the definition (\ref{eta})}.
\end{eqnarray*}
\end{proof}

\section {Integral closures in modular function fields}

In this section, we shall find explicit generators of $\mathcal{R}_0(2^n)$ over $\mathbb{C}$ by using eta-quotients.
As a consequence of this result, we shall further show that every modular form on $\Gamma_0(2^n)$ can be written as a sum of eta-quotients.

\begin{lemma}\label{closure}
Let $N$ be a positive integer.
The ring $\mathcal{R}_0(N)$ consists of weakly holomorphic \textup{(}that is, holomorphic on $\mathbb{H}$\textup{)} functions in $\mathbb{C}(X_0(N))$.
\end{lemma}
\begin{proof}
Note first that every weakly holomorphic function in $\mathbb{C}(X_0(1))=\mathbb{C}(X(1))$ is
a polynomial in $j(\tau)$ over $\mathbb{C}$ \cite[$\S$5.2, Theorem 2]{Lang}.
\par
Let $h(\tau)\in\mathcal{R}_0(N)$, so $h(\tau)$ satisfies
\begin{equation*}
h(\tau)^m+P_{m-1}(j(\tau))h(\tau)^{m-1}+\cdots+P_0(j(\tau))=0
\end{equation*}
for some $m\geq1$ and $P_{m-1}(X),
\ldots,P_0(X)\in\mathbb{C}[X]$. Dividing both sides by $h(\tau)^m$ we get
\begin{equation}\label{equation}
1+P_{m-1}(j(\tau))(1/h(\tau))+\cdots+P_0(j(\tau))(1/h(\tau))^m=0.
\end{equation}
Suppose that $h(\tau)$ has a pole at a point $\tau_0$ in $\mathbb{H}$, so
$1/h(\tau)$ has a zero at the point.
But, if we insert $\tau=\tau_0$ into (\ref{equation}), then
we get a contradiction $1=0$. Hence $h(\tau)$ must be weakly holomorphic.
\par
Conversely, let $h(\tau)\in\mathbb{C}(X_0(N))$
be weakly holomorphic.
Since $\Gamma(N)\leq\Gamma_0(N)\leq\Gamma(1)$,
$\mathbb{C}(X_0(N))$ is an intermediate field of the extension
$\mathbb{C}(X(N))/\mathbb{C}(X(1))$.
Thus
$h(\tau)\circ\gamma$ for $\gamma\in\Gamma(1)$
represent all the Galois conjugates of $h(\tau)$ over $\mathbb{C}(X(1))=\mathbb{C}(X_0(1))$. It follows that
every coefficient of $\min(h(\tau),\mathbb{C}(X_0(1)))$ is also weakly holomorphic, and hence belongs to $\mathbb{C}[j(\tau)]$. This shows that $h(\tau)\in\mathcal{R}_0(N)$, and completes the proof.
\end{proof}

For each $n\geq3$ we define
\begin{equation}\label{hn}
h_n(\tau)=\frac{\wp_{\left[\begin{smallmatrix}
1/2\\1/2\end{smallmatrix}\right]}(2^{n-1}\tau)-
\wp_{\left[\begin{smallmatrix}
0\\1/2\end{smallmatrix}\right]}(2^{n-1}\tau)}
{\wp_{\left[\begin{smallmatrix}
1/2\\1/2\end{smallmatrix}\right]}(2^{n-2}\tau)-
\wp_{\left[\begin{smallmatrix}
0\\1/2\end{smallmatrix}\right]}(2^{n-2}\tau)}\quad(\tau\in\mathbb{H}).
\end{equation}
It belongs to $\mathbb{C}(X_0(2^n))$ by Lemma \ref{plevel}(ii), and
has neither zeros nor poles on $\mathbb{H}$
by (\ref{pequal}). Hence it is in $\mathcal{R}_0(2^n)^\times$ by Lemma \ref{closure}.

\begin{lemma}\label{htoh}
We have
$h_n(\tau)=\eta(2^{n-2}\tau)^{12}/\eta(2^{n-1}\tau)^4\eta(2^{n-3}\tau)^8$.
\end{lemma}
\begin{proof}
We find that
\begin{eqnarray*}
h_n(\tau)&=&\frac{\wp_{\left[\begin{smallmatrix}
1/2\\1/2\end{smallmatrix}\right]}(2^{n-1}\tau)-
\wp_{\left[\begin{smallmatrix}
0\\1/2\end{smallmatrix}\right]}(2^{n-1}\tau)}
{\wp_{\left[\begin{smallmatrix}
1/2\\1/2\end{smallmatrix}\right]}(2^{n-2}\tau)-
\wp_{\left[\begin{smallmatrix}
0\\1/2\end{smallmatrix}\right]}(2^{n-2}\tau)}\\
&=&
\frac{g_{\left[\begin{smallmatrix}1/2\\1\end{smallmatrix}\right]}(2^{n-1}\tau)
g_{\left[\begin{smallmatrix}1/2\\0\end{smallmatrix}\right]}(2^{n-1}\tau)\eta(2^{n-1}\tau)^4
/g_{\left[\begin{smallmatrix}1/2\\1/2\end{smallmatrix}\right]}(2^{n-1}\tau)^2
g_{\left[\begin{smallmatrix}0\\1/2\end{smallmatrix}\right]}(2^{n-1}\tau)^2}
{g_{\left[\begin{smallmatrix}1/2\\1\end{smallmatrix}\right]}(2^{n-2}\tau)
g_{\left[\begin{smallmatrix}1/2\\0\end{smallmatrix}\right]}(2^{n-2}\tau)\eta(2^{n-2}\tau)^4/
g_{\left[\begin{smallmatrix}1/2\\1/2\end{smallmatrix}\right]}(2^{n-2}\tau)^2
g_{\left[\begin{smallmatrix}0\\1/2\end{smallmatrix}\right]}(2^{n-2}\tau)^2}
\quad\textrm{by Lemma \ref{gtoh}(i)}\\
&=&
\frac{g_{\left[\begin{smallmatrix}1/2\\0\end{smallmatrix}\right]}(2^{n-1}\tau)^2\eta(2^{n-1}\tau)^4
/g_{\left[\begin{smallmatrix}1/2\\1/2\end{smallmatrix}\right]}(2^{n-1}\tau)^2
g_{\left[\begin{smallmatrix}0\\1/2\end{smallmatrix}\right]}(2^{n-1}\tau)^2}
{
g_{\left[\begin{smallmatrix}1/2\\0\end{smallmatrix}\right]}(2^{n-2}\tau)^2\eta(2^{n-2}\tau)^4/
g_{\left[\begin{smallmatrix}1/2\\1/2\end{smallmatrix}\right]}(2^{n-2}\tau)^2
g_{\left[\begin{smallmatrix}0\\1/2\end{smallmatrix}\right]}(2^{n-2}\tau)^2
}
\quad\textrm{by (\ref{translation})}\\
&=&\frac{g_{\left[\begin{smallmatrix}1/2\\0\end{smallmatrix}\right]}(2^{n-1}\tau)^4\eta(2^{n-1}\tau)^4}
{g_{\left[\begin{smallmatrix}1/2\\0\end{smallmatrix}\right]}(2^{n-2}\tau)^4\eta(2^{n-2}\tau)^4}
\quad\textrm{by Lemma \ref{gtoh}(ii)}\\
&=&\frac{\eta(2^{n-2}\tau)^{12}}{\eta(2^{n-1}\tau)^4\eta(2^{n-3}\tau)^8}\quad\textrm{by Lemma \ref{gtoh}(iii)}.
\end{eqnarray*}
\end{proof}

\begin{remark}\label{modremark}
Due to Lemma \ref{htoh}, one can also use Lemma \ref{modularity} to show that
$h_n(\tau)$ belongs to $\mathcal{R}_0(2^n)^\times$ and has rational Fourier coefficients
with respect to $q$.
\end{remark}

\begin{lemma}\label{next}
Let $n\geq3$.
We have
$\mathcal{R}_0(2^n)=\mathcal{R}_0(2^{n-1})[h_n(\tau)]$.
\end{lemma}
\begin{proof}
Since $h_n(\tau)\in\mathcal{R}_0(2^n)^\times$,
we obviously have $\mathcal{R}_0(2^{n-1})[h_n(\tau)]\subseteq\mathcal{R}_0(2^n)$.
\par
Note that
\begin{equation*}
\mathrm{Gal}(\mathbb{C}(X_0(2^n))/\mathbb{C}(X_0(2^{n-1})))
\simeq\overline{\Gamma}_0(2^{n-1})/
\overline{\Gamma}_0(2^n)=\{I_2,\left[\begin{matrix}1&0\\2^{n-1}&1\end{matrix}\right]\},
\end{equation*}
where $\overline{\Gamma}_0(N)=\langle\Gamma_0(N),\pm I_2\rangle/\{\pm I_2\}$ ($N\geq1$)
\cite[p.31]{Shimura}.
Let $\mathbf{v}=
\left[\begin{matrix}v_1\\v_2\end{matrix}\right]=
\left[\begin{matrix}1/2\\1/2\end{matrix}\right]$ or
$\left[\begin{matrix}0\\1/2\end{matrix}\right]$.
Since $\wp_\mathbf{v}(2^{n-2}\tau)$ is a meromorphic modular forms of weight $2$ on
$\Gamma_0(2^{n-1})$ by Lemma \ref{plevel}(ii), we have
\begin{equation*}
\wp_\mathbf{v}(2^{n-2}\tau)\circ
\left[\begin{matrix}1&0\\2^{n-1}&1\end{matrix}\right]=
(2^{n-1}\tau+1)^2\wp_\mathbf{v}(2^{n-2}\tau).
\end{equation*}
Furthermore, we get by Lemma \ref{plevel}(i) that
\begin{equation*}
\wp_\mathbf{v}(2^{n-1}\tau)\circ\left[\begin{matrix}1&0\\2^{n-1}&1\end{matrix}\right]=
(2^{n-1}\tau+1)^2\wp_{\left[\begin{smallmatrix}v_1+v_2\\v_2\end{smallmatrix}\right]}(2^{n-1}\tau).
\end{equation*}
Thus we obtain
\begin{equation}\label{hconjugate}
h_n(\tau)\circ\left[\begin{matrix}1&0\\2^{n-1}&1\end{matrix}\right]
=-h_n(\tau),
\end{equation}
and hence $\mathbb{C}(X_0(2^n))=\mathbb{C}(X_0(2^{n-1}))(h_n(\tau))$. Furthermore,
the fact $h_n(\tau)\in\mathcal{R}_0(2^n)^\times$ and (\ref{hconjugate}) imply that
\begin{equation}\label{square1}
h_n(\tau)^2\in\mathcal{R}_0(2^{n-1})^\times.
\end{equation}
\par
Now, let $h(\tau)\in\mathcal{R}_0(2^n)$. Since $\mathbb{C}(X_0(2^n))$
is a quadratic extension of
$\mathbb{C}(X_0(2^{n-1}))$ generated by $h_n(\tau)$,
we can express $h(\tau)$ as
\begin{equation}\label{c0c1}
h(\tau)=c_0(\tau)+c_1(\tau)h_n(\tau)~\textrm{for some}~c_0(\tau),c_1(\tau)\in \mathbb{C}(X_0(2^{n-1})).
\end{equation}
Put $h'(\tau)=h(\tau)\circ\left[\begin{matrix}1&0\\2^{n-1}&1\end{matrix}\right]$,
which also lies in $\mathcal{R}_0(2^n)$ by Lemma \ref{closure}. By (\ref{hconjugate}) and (\ref{c0c1}) we get a system
\begin{equation*}
\left[\begin{matrix}
h(\tau)\\h'(\tau)
\end{matrix}\right]=\left[\begin{matrix}
1 & h_n(\tau)\\
1 & -h_n(\tau)\\
\end{matrix}\right]\left[\begin{matrix}c_0(\tau)\\c_1(\tau)
\end{matrix}\right]
\end{equation*}
and find
\begin{equation*}
c_0(\tau)=(h(\tau)+h'(\tau))/2~\textrm{and}~
c_1(\tau)=(h(\tau)h_n(\tau)+h'(\tau)(-h_n(\tau)))/2h_n(\tau)^2,
\end{equation*}
which belong to $\mathcal{R}_0(2^{n-1})$ by (\ref{hconjugate}) and (\ref{square1}).
This shows that $h(\tau)\in\mathcal{R}_0(2^{n-1})[h_n(\tau)]$, and hence
$\mathcal{R}_0(2^n)\subseteq\mathcal{R}_0(2^{n-1})[h_n(\tau)]$.
Therefore we achieve $\mathcal{R}_0(2^n)=
\mathcal{R}_0(2^{n-1})[h_n(\tau)]$, as desired.
\end{proof}

Let $\mathbb{Q}(X_0(N))$ be the subfield of $\mathbb{C}(X_0(N))$ consisting of functions
with rational Fourier coefficients with respect to $q$.
Let
\begin{equation*}
g_{0,4}(\tau)=\eta(4\tau)^8/\eta(\tau)^8,
\end{equation*}
which belongs to $\mathbb{Q}(X_0(4))$ by Lemma \ref{modularity}.

\begin{lemma}\label{level4}
We have the following structures on $\Gamma_0(4)$:
\begin{itemize}
\item[\textup{(i)}]
$\mathbb{C}(X_0(4))=\mathbb{C}(g_{0,4}(\tau))$ and
$\mathbb{Q}(X_0(4))=\mathbb{Q}(g_{0,4}(\tau))$.
\item[\textup{(ii)}]
$\mathcal{R}_0(4)=\mathbb{C}[g_{0,4}(\tau),\eta(2\tau)^{24}/\eta(4\tau)^{16}\eta(\tau)^8,
\eta(4\tau)^{16}\eta(\tau)^8/\eta(2\tau)^{24}]$.
\end{itemize}
\end{lemma}
\begin{proof}
(i) See \cite[Table 2 and Lemma 4.1]{K-S}  and \cite[Remark 3.4]{E-K-S}.\\
(ii) See \cite[Theorem 3.3(i) and Remark 3.4]{E-K-S}.
\end{proof}

\begin{theorem}\label{generators}
Let $n\geq3$. We have
\begin{equation*}
\mathcal{R}_0(2^n)=\mathbb{C}[g_{0,4}(\tau),\eta(2\tau)^{24}/\eta(4\tau)^{16}\eta(\tau)^8,
\eta(4\tau)^{16}\eta(\tau)^8/\eta(2\tau)^{24},h_3(\tau),
\ldots,h_n(\tau)].
\end{equation*}
\end{theorem}
\begin{proof}
This result follows from Lemmas \ref{next} and \ref{level4}(ii).
\end{proof}

\begin{corollary}\label{sum}
Every modular form on $\Gamma_0(2^n)$ \textup{($n\geq2$)}
can be expressed as a sum of eta-quotient.
\end{corollary}
\begin{proof}
Let $h(\tau)$ be a modular form of weight $2k$ ($k\geq0$) on $\Gamma_0(2^n)$.
Since $\eta(2\tau)^4/\eta(4\tau)^8$ is a meromorphic modular form of weight $-2$ on $\Gamma_0(4)$
and is weakly holomorphic by Lemma \ref{modularity} and Remark \ref{no}, the product
$h(\tau)(\eta(2\tau)^4/\eta(4\tau)^8)^k$ belongs to $\mathcal{R}_0(2^n)$. Thus it can be
expressed as a sum of eta-quotients by Lemma \ref{level4}(ii) and Theorem \ref{generators}, and hence
$h(\tau)$ itself is a sum of eta-quotients, too.
\end{proof}

\section {Generation of ring class fields}

For a positive integer $N$ let $\mathcal{F}_N$ be the field of
functions in $\mathbb{C}(X(N))$ whose Fourier coefficients with
respect to $q^{1/N}$ lie in $\mathbb{Q}(\zeta_N)$, where $\zeta_N=e^{2\pi i/N}$.
It is well-known that $\mathcal{F}_N$ is a Galois extension of $\mathcal{F}_1=\mathbb{Q}(j(\tau))$
whose Galois group is isomorphic to
\begin{equation*}
\mathrm{GL}_2(\mathbb{Z}/N\mathbb{Z})/\{\pm I_2\}=
\{\left[\begin{matrix}1&0\\0&d\end{matrix}\right]~|~d\in(\mathbb{Z}/N\mathbb{Z})^\times\}
\cdot\mathrm{SL}_2(\mathbb{Z}/N\mathbb{Z})/\{\pm I_2\}.
\end{equation*}
Let $h(\tau)=\sum_{n>-\infty}c_nq^{n/N}$ be a function in $\mathcal{F}_N$.
The matrix $\left[\begin{matrix}1&0\\0&d\end{matrix}\right]$ with $d\in(\mathbb{Z}/N\mathbb{Z})^\times$ acts on
$h(\tau)$ by
\begin{equation}\label{d}
h(\tau)^{\left[\begin{smallmatrix}1&0\\0&d\end{smallmatrix}\right]}=
\sum_{n>-\infty}c_n^{\sigma_d}q^{n/N},
\end{equation}
where $\sigma_d$ is the automorphism of $\mathbb{Q}(\zeta_N)$
given by
$\zeta_N^{\sigma_d}=\zeta_N^d$.
And, $\alpha\in\mathrm{SL}_2(\mathbb{Z}/N\mathbb{Z})/\{\pm I_2\}$ acts on $h(\tau)$ by
\begin{equation}\label{SL}
h(\tau)^\alpha=h(\tau)\circ\widetilde{\alpha},
\end{equation}
where $\widetilde{\alpha}$ is any
preimage of $\alpha$ with respect to the natural reduction
$\mathrm{SL}_2(\mathbb{Z})\rightarrow\mathrm{SL}_2(\mathbb{Z}/N\mathbb{Z})/\{\pm I_2\}$
\cite[$\S$6.3, Theorem 3]{Lang}.

Let $K$ be an imaginary quadratic field with $\tau_K$ as in (\ref{tauK}).
We denote by $K_{N\mathcal{O}_K}$ its ray class field modulo $N\mathcal{O}_K$.
Furthermore, we let $H_K$ be its Hilbert class field and $H_{K,N}$ be the ring class field of the order of conductor $N$ in $K$.
As consequences of the main theorem of the theory of complex multiplication we obtain
\begin{eqnarray}
K_{N\mathcal{O}_K}&=&K(h(\tau_K)~|~h(\tau)\in\mathcal{F}_N~\textrm{is finite at}~\tau_K),
\nonumber\\
H_{N,K}&=&K(h(\tau_K)~|~h(\tau)\in\mathbb{Q}(X_0(N))~\textrm{is finite at}~\tau_K)\label{specialization2}
\end{eqnarray}
(\cite[$\S$10.1, Corollary to Theorem 2]{Lang} and\cite[Theorem 3.4]{K-S}).

\begin{lemma}[Shimura's reciprocity law]\label{Shimura}
Let $\min(\tau_K,\mathbb{Q})=X^2+BX+C$ and
\begin{equation*}
W_{K,N}=\{\gamma=\left[\begin{matrix}t-Bs & -Cs\\s&t\end{matrix}\right]~|~
t,s\in\mathbb{Z}/N\mathbb{Z}~\textrm{such that}~\gamma\in\mathrm{GL}_2(\mathbb{Z}/N\mathbb{Z})\}.
\end{equation*}
\begin{itemize}
\item[\textup{(i)}]
The group $W_{K,N}$
gives rise to the surjection
\begin{eqnarray*}
W_{K,N}&\rightarrow&\mathrm{Gal}(K_{N\mathcal{O}_K}/H_K)\\
\gamma&\mapsto&(h(\tau_K)\mapsto h^\gamma(\tau_K)~|~h(\tau)\in\mathcal{F}_N~
\textrm{is finite at}~\tau_K)
\end{eqnarray*}
whose kernel is
\begin{equation*}
\left\{\begin{array}{ll}
\{\pm I_2,\pm\left[\begin{smallmatrix}0 & -1\\1&0\end{smallmatrix}\right]\} & \textrm{if}~K=\mathbb{Q}(\sqrt{-1}),\\
\{\pm I_2,\pm\left[\begin{smallmatrix}-1 & -1\\1&0\end{smallmatrix}\right],
 \pm\left[\begin{smallmatrix}0 &-1\\1&1\end{smallmatrix}\right]\}& \textrm{if}~K=\mathbb{Q}(\sqrt{-3}),\\
\{\pm I_2\} & \textrm{otherwise}.
\end{array}\right.
\end{equation*}
\item[\textup{(ii)}] The subgroup $\{tI_2~|~t\in(\mathbb{Z}/N\mathbb{Z})^\times\}$
of $W_{K,N}$ gives rise to the isomorphism
\begin{eqnarray*}
\{tI_2~|~t\in(\mathbb{Z}/N\mathbb{Z})^\times\}/\{\pm I_2\}&\rightarrow&\mathrm{Gal}(K_{N\mathcal{O}_K}/H_{K,N})\\
tI_2&\mapsto&(h(\tau_K)\mapsto h^{tI_2}(\tau_K)~|~h(\tau)\in\mathcal{F}_N~\textrm{is finite at}~
\tau_K).
\end{eqnarray*}
\end{itemize}
\end{lemma}
\begin{proof}
(i) See \cite[Theorem 6.31 and Proposition 6.34]{Shimura} and \cite[$\S$3]{Stevenhagen}.\\
(ii) See \cite[Proposition 3.8]{K-S2}.
\end{proof}

\begin{lemma}\label{degree}
Let $\mathcal{O}=[N\tau_K,1]$ be the order of conductor $N$ in $K$.
We have the degree formula
\begin{equation*}
[H_{K,N}:K]=\frac{[H_K:K]N}{[\mathcal{O}_K^\times:\mathcal{O}^\times]}
\prod_{p|N}\bigg(1-\bigg(\frac{d_K}{p}\bigg)\frac{1}{p}\bigg),
\end{equation*}
where
\begin{equation*}
\bigg(\frac{d_K}{p}\bigg)=\left\{\begin{array}{ll}
\textrm{the Legendre symbol} & \textrm{if $p$ is an odd prime},\\
\textrm{the Kronecker symbol} & \textrm{if}~p=2.
\end{array}\right.
\end{equation*}
\end{lemma}
\begin{proof}
See \cite[Theorem 7.24]{Cox}.
\end{proof}

\begin{lemma}\label{jg}
Let $N\equiv0\pmod{4}$.
There exist polynomials $A(X),B(X)\in\mathbb{Q}[X]$ satisfying
\begin{itemize}
\item[\textup{(i)}] $j(N\tau)=A(g_{0,4}((N/4)\tau))/B(g_{0,4}((N/4)\tau))$,
\item[\textup{(ii)}] $B(g_{0,4}((N/4)\tau_K))\neq0$.
\end{itemize}
\end{lemma}
\begin{proof}
Since
$j(4\tau)\in\mathbb{Q}(X_0(4))$ and $\mathbb{Q}(X_0(4))=\mathbb{Q}(g_{0,4}(\tau))$
by Lemma \ref{level4}(i),
we can express $j(4\tau)$ as
$j(4\tau)=A(g_{0,4}(\tau))/B(g_{0,4}(\tau))$
for some relatively prime polynomials $A(X),B(X)\in\mathbb{Q}[X]$.
It follows that $j(N\tau)=A(g_{0,4}((N/4)\tau))/B(g_{0,4}((N/4)\tau))$.
\par
On the other hand, one can readily check by Lemma
\ref{modularity} and Remark \ref{no} that the function
\begin{equation*}
g_{0,4}((N/4)\tau)=\eta(N\tau)^8/\eta((N/4)\tau)^8
\end{equation*}
belongs to $\mathbb{Q}(X_0(N))$ and is weakly holomorphic. Thus
we obtain by (\ref{specialization2}) that
\begin{equation}\label{g(NtauK)}
g_{0,4}((N/4)\tau_K)\in H_{K,N}.
\end{equation}
\par
Suppose that $B(g_{0,4}((N/4)\tau_K))=0$.
Since $j(N\tau)$ is weakly holomorphic, we must have $A(g_{0,4}((N/4)\tau_K))=0$.
But this implies that
$\min(g_{0,4}((N/4)\tau_K),\mathbb{Q})$ divides both $A(X)$ and $B(X)$,
which yields a contradiction.
Therefore $B(g_{0,4}((N/4)\tau_K))$ is nonzero.
\end{proof}

\begin{lemma}\label{integrality}
If $M$ is a positive integer and $\tau_0\in\mathbb{H}$ is an imaginary quadratic
argument, then the special value
$M\eta(M\tau_0)^2/\eta(\tau_0)^2$ is an algebraic integer dividing $M$.
\end{lemma}
\begin{proof}
See \cite[$\S$12.2, Theorem 4]{Lang}.
\end{proof}

\begin{theorem}\label{ringclassinvariant}
Let $N\equiv0\pmod{4}$.
Then the special value $256\eta(N\tau_K)^8/\eta((N/4)\tau_K)^8$ generates $H_{K,N}$ over $K$
as a real algebraic integer.
\end{theorem}
\begin{proof}
Let $A(X)$ and $B(X)$ be polynomials in $\mathbb{Q}[X]$ satisfying (i) and (ii) in Lemma \ref{jg}. We deduce that
\begin{eqnarray*}
H_{K,N}&=&K(j(N\tau_K))\quad\textrm{by (\ref{j(NtauK)})}\\
&=& K(A(g_{0,4}((N/4)\tau_K))/B(g_{0,4}((N/4)\tau_K)))\quad\textrm{by Lemma \ref{jg}}\\
&\subseteq&K(g_{0,4}((N/4)\tau_K))\quad\textrm{because $A(X),B(X)\in\mathbb{Q}[X]$}\\
&\subseteq&H_{K,N}\quad\textrm{by (\ref{g(NtauK)})}.
\end{eqnarray*}
Therefore we achieve $H_{N,K}=K(g_{0,4}((N/4)\tau_K))=K(\eta(N\tau_K)^8/\eta((N/4)\tau_K)^8)$.
Moreover, $256\eta(N\tau_K)^8/\eta((N/4)\tau_K)^8$ is an algebraic integer by Lemma
\ref{integrality} with $M=4$ and $\tau_0=(N/4)\tau_K$.
\end{proof}

\begin{remark}
Since $256\eta(N\tau_K)^8/\eta((N/4)\tau_K)^8$
is a real algebraic integer, its minimal polynomial over $K$ has integer coefficients
\cite[Remark 4.9]{K-S}.
\end{remark}

\begin{example}
Consider the case when $K=\mathbb{Q}(\sqrt{-7})$ and $N=12$.
Let $h(\tau)=256\eta(12\tau)^8/
\eta(3\tau)^8$.
Then the special value
$h(\tau_K)$ generates $H_{K,N}$ over $K$
as a real algebraic integer
by Theorem \ref{ringclassinvariant}.
\par
We have $H_K=K$ \cite[Theorem 12.34]{Cox} and
find by Lemma \ref{Shimura} that
\begin{eqnarray*}
\mathrm{Gal}(H_{K,N}/H_K)&\simeq&\{
\left[\begin{matrix}1&0\\0&1\end{matrix}\right]
\left[\begin{matrix}1&0\\0&1\end{matrix}\right],
\left[\begin{matrix}1&0\\0&7\end{matrix}\right]
\left[\begin{matrix}-1&-4\\2&7\end{matrix}\right],
\left[\begin{matrix}1&0\\0&5\end{matrix}\right]
\left[\begin{matrix}-3&4\\-4&5\end{matrix}\right],
\left[\begin{matrix}1&0\\0&7\end{matrix}\right]
\left[\begin{matrix}1&0\\6&1\end{matrix}\right],\\
&&
\left[\begin{matrix}1&0\\0&1\end{matrix}\right]
\left[\begin{matrix}5&8\\8&13\end{matrix}\right],
\left[\begin{matrix}1&0\\0&11\end{matrix}\right]
\left[\begin{matrix}-9&4\\2&-1\end{matrix}\right],
\left[\begin{matrix}1&0\\0&11\end{matrix}\right]
\left[\begin{matrix}1&-4\\-2&9\end{matrix}\right],
\left[\begin{matrix}1&0\\0&5\end{matrix}\right]
\left[\begin{matrix}11&4\\8&3\end{matrix}\right]\}.
\end{eqnarray*}
Now, due to Lemma \ref{Shimura}, (\ref{d}) and (\ref{SL}) one can compute
(by using Maple ver.15)
\begin{eqnarray*}
\min(h(\tau_K),K)&=&
(X-h\circ\left[\begin{matrix}1&0\\0&1\end{matrix}\right](\tau_K))
(X-h\circ\left[\begin{matrix}-1&-4\\2&7\end{matrix}\right](\tau_K))
(X-h\circ\left[\begin{matrix}-3&4\\-4&5\end{matrix}\right](\tau_K))\\
&&(X-h\circ\left[\begin{matrix}1&0\\6&1\end{matrix}\right](\tau_K))
(X-h\circ\left[\begin{matrix}5&8\\8&13\end{matrix}\right](\tau_K))
(X-h\circ\left[\begin{matrix}-9&4\\2&-1\end{matrix}\right](\tau_K))\\
&&(X-h\circ\left[\begin{matrix}1&-4\\-2&9\end{matrix}\right](\tau_K))
(X-h\circ\left[\begin{matrix}11&4\\8&3\end{matrix}\right](\tau_K))
\\
&=&
X^8+64X^7+2365X^6+5617X^5+1025614X^4+13744576X^3\\
&&+99275140X^2+263731264X+1.
\end{eqnarray*}
\end{example}

\begin{theorem}\label{2integers}
Let $n\geq3$.
We have
\begin{equation*}
\mathcal{O}_{H_{K,2^n}}[1/2]=
\mathcal{O}_{H_{K,4}}[1/2,h_3(\tau_K),\ldots,h_n(\tau_K)].
\end{equation*}
\end{theorem}
\begin{proof}
For each $m\geq3$ we see by Lemma \ref{htoh}
\begin{equation*}
h_m(\tau)=(\eta(2^{m-2}\tau)/\eta(2^{m-1}\tau))^4(\eta(2^{m-2}\tau)/\eta(2^{m-3}\tau))^8.
\end{equation*}
Since $h_m(\tau)\in\mathcal{R}_0(2^m)^\times\cap\mathbb{Q}(X_0(2^m))$ by Remark \ref{modremark}, we obtain
by (\ref{specialization2})
and Lemma \ref{integrality} that
\begin{equation}\label{1/2in}
h_m(\tau_K)\in(\mathcal{O}_{H_{K,2^m}}[1/2])^\times.
\end{equation}
So we deduce the inclusion
\begin{equation*}
\mathcal{O}_{H_{K,2^n}}\supseteq\mathcal{O}_{H_{K,4}}[1/2,h_3(\tau_K),\ldots,h_n(\tau_K)].
\end{equation*}
\par
On the other hand, we see by Lemma \ref{degree} that $[H_{K,2^m}:H_{K,2^{m-1}}]=2$,
and find by Lemma \ref{Shimura} that
\begin{equation*}
\mathrm{Gal}(H_{K,2^m}/H_{K,2^{m-1}})\simeq\{I_2,\left[\begin{matrix}1-B2^{m-1} & -C2^{m-1}\\
2^{m-1} & 1\end{matrix}\right]\},
\end{equation*}
where $\min(\tau_K,\mathbb{Q})=X^2+BX+C$.
Decompose $\left[\begin{matrix}1-B2^{m-1} & -C2^{m-1}\\
2^{m-1} & 1\end{matrix}\right]\in\mathrm{GL}_2(\mathbb{Z}/2^m\mathbb{Z})$ into
\begin{equation*}
\left[\begin{matrix}1-B2^{m-1} & -C2^{m-1}\\
2^{m-1} & 1\end{matrix}\right]\equiv
\left[\begin{matrix}1&0\\0&1-B2^{m-1}\end{matrix}\right]
\left[\begin{matrix}
1-B2^{m-1}+a2^m & -C2^{m-1}+b2^m\\
2^{m-1}+c2^m & 1-B2^{m-1}+d2^m
\end{matrix}\right]\pmod{2^m}
\end{equation*}
for some $a,b,c,d\in\mathbb{Z}$ in order that $\left[\begin{matrix}
1-B2^{m-1}+a2^m & -C2^{m-1}+b2^m\\
2^{m-1}+c2^m & 1-B2^{m-1}+d2^m
\end{matrix}\right]\in\mathrm{SL}_2(\mathbb{Z})$.
We derive that
\begin{eqnarray*}
&&h_m(\tau_K)^{\left[\begin{smallmatrix}1-B2^{m-1} & -C2^{m-1}\\
2^{m-1} & 1\end{smallmatrix}\right]}\\
&=&h_m(\tau)^{\left[\begin{smallmatrix}1 & 0\\
0 & 1-B2^{m-1}\end{smallmatrix}\right]
\left[\begin{smallmatrix}
1-B2^{m-1}+a2^m & -C2^{m-1}+b2^m\\
2^{m-1}+c2^m & 1-B2^{m-1}+d2^m
\end{smallmatrix}\right]}(\tau_K)\quad\textrm{by Lemma \ref{Shimura}(i)}\\
&=&h_m(\tau)^{\left[\begin{smallmatrix}
1-B2^{m-1}+a2^m & -C2^{m-1}+b2^m\\
2^{m-1}+c2^m & 1-B2^{m-1}+d2^m
\end{smallmatrix}\right]
}(\tau_K)\quad\textrm{by the fact $h_m(\tau)\in\mathbb{Q}(X_0(2^m))$ and (\ref{d})}\\
&=&
\frac{\wp_{\left[\begin{smallmatrix}
1/2\\1/2\end{smallmatrix}\right]}(2^{m-1}\tau)-
\wp_{\left[\begin{smallmatrix}
0\\1/2\end{smallmatrix}\right]}(2^{m-1}\tau)}
{\wp_{\left[\begin{smallmatrix}
1/2\\1/2\end{smallmatrix}\right]}(2^{m-2}\tau)-
\wp_{\left[\begin{smallmatrix}
0\\1/2\end{smallmatrix}\right]}(2^{m-2}\tau)}\circ
\left[\begin{matrix}
1-B2^{m-1}+a2^m & -C2^{m-1}+b2^m\\
2^{m-1}+c2^m & 1-B2^{m-1}+d2^m
\end{matrix}\right]
(\tau_K)\\
&&\quad\textrm{by the definition (\ref{hn}) and (\ref{SL})}\\
&=&\frac{\wp_{\left[\begin{smallmatrix}
0\\1/2\end{smallmatrix}\right]}(2^{m-1}\tau_K)-
\wp_{\left[\begin{smallmatrix}
1/2\\1/2\end{smallmatrix}\right]}(2^{m-1}\tau_K)}
{\wp_{\left[\begin{smallmatrix}
1/2\\1/2\end{smallmatrix}\right]}(2^{m-2}\tau_K)-
\wp_{\left[\begin{smallmatrix}
0\\1/2\end{smallmatrix}\right]}(2^{m-2}\tau_K)}\quad\textrm{by Lemma \ref{plevel}}\\
&=&-h_m(\tau_K).
\end{eqnarray*}
This shows that
$H_{K,2^m}=H_{K,2^{m-1}}(h_m(\tau_K))$, and
$-h_m(\tau_K)$ is another Galois conjugate of $h_m(\tau_K)$ over $H_{K,2^{m-1}}$.
Thus we get by (\ref{1/2in}) that
\begin{equation}\label{square2}
h_m(\tau_K)^2\in(\mathcal{O}_{H_{K,2^{m-1}}}[1/2])^\times.
\end{equation}
\par
Now, let $x$ be an element of $\mathcal{O}_{H_{K,2^m}}[1/2]$
and $x'\in\mathcal{O}_{H_{K,2^m}}[1/2]$ be its Galois conjugate (not necessarily distinct).
Express $x$ as $x=c_0+c_1h_m(\tau_K)$ for some $c_0,c_1\in H_{K,2^{m-1}}$.
Since $x'=c_0+c_1(-h_m(\tau_K))$, we get a system
\begin{equation*}
\left[\begin{matrix}x\\x'\end{matrix}\right]=\left[\begin{matrix}
1 & h_m(\tau_K)\\
1 & -h_m(\tau_K)
\end{matrix}\right]
\left[\begin{matrix}
c_0\\c_1
\end{matrix}\right].
\end{equation*}
It follows that
\begin{equation*}
c_0=(x+x')/2~\textrm{and}~
c_1=(xh_m(\tau_K)+x'(-h_m(\tau_K)))/2h_m(\tau_K)^2,
\end{equation*}
which belong to $\mathcal{O}_{H_{K,2^{m-1}}}[1/2]$ by (\ref{square2}).
This shows that $x\in\mathcal{O}_{H_{K,2^{m-1}}}[1/2,h_m(\tau_K)]$, and hence
$\mathcal{O}_{H_{K,2^m}}[1/2]\subseteq\mathcal{O}_{H_{K,2^{m-1}}}[1/2,h_m(\tau_K)]$. Hence we achieve
\begin{equation*}
\mathcal{O}_{H_{K,2^n}}[1/2]\subseteq\mathcal{O}_{H_{K,4}}[1/2,h_3(\tau_K),\ldots,h_n(\tau_K)].
\end{equation*}
This completes the proof.
\end{proof}

\bibliographystyle{amsplain}

\begin{thebibliography}{10}

\bibitem {Choi} D. Choi, \textit{Spaces of modular forms generated by eta-quotients},
Ramanujan J. 14 (2007), no. 1, 69--77.

\bibitem {Cox} D. A. Cox, \textit{Primes of the form $x^2+ny^2$: Fermat, Class Field, and Complex Multiplication},
John Wiley \& Sons, Inc., New York, 1989.

\bibitem {E-K-S} I. S. Eum, J. K. Koo and D. H. Shin,
\textit{Some applications of modular units}, to appear in Proc. Edinb. Math. Soc.,
http://arxiv.org/abs/1207.1609.

\bibitem {K-L} D. Kubert and S. Lang, \textit{Modular Units}, Grundlehren der mathematischen Wissenschaften 244, Spinger-Verlag, New York-Berlin, 1981.

\bibitem {K-S} J. K. Koo and D. H. Shin, \textit{On some arithmetic properties of Siegel functions}, Math. Zeit. 264 (2010),  no. 1, 137--177.

\bibitem {K-S2} J. K. Koo and D. H. Shin, \textit{Function fields of certain arithmetic curves and application}, Acta Arith. 141 (2010), no. 4, 321--334.

\bibitem {K-S-Y} J. K. Koo, D. H. Shin and D. S. Yoon,
\textit{Algebraic integers as special values of modular units}, Proc. Edinb. Math. Soc. (2)  55 (2012),  no. 1, 167--179.

\bibitem {Lang} S. Lang, \textit{Elliptic Functions}, 2nd edn, Grad. Texts in Math.
112, Spinger-Verlag, New York, 1987.

\bibitem {Ono} K. Ono, \textit{The Web of Modularity:
Arithmetic of the Coefficients of Modular Forms and $q$-series},
CBMS Regional Conf. Series in Math. 102,
Amer. Math. Soc., Providence, R. I., 2004.

\bibitem {Shimura} G. Shimura, \textit{Introduction to the Arithmetic Theory of Automorphic Functions}, Iwanami Shoten and Princeton
University Press, Princeton, N. J., 1971.

\bibitem {Silverman} J. H. Silverman, \textit{The Arithmetic of Elliptic Curves},
Grad. Texts in Math. 106,
Springer-Verlag, New York, 1992.

\bibitem {Stevenhagen} P. Stevenhagen, \textit{Hilbert's 12th problem, complex multiplication and Shimura reciprocity}, Class Field Theory-Its Centenary and Prospect (Tokyo,
1998), 161--176, Adv. Stud. Pure Math. 30, Math. Soc. Japan, Tokyo,
2001.

\end{thebibliography}

\address{
Research Institute of Mathematics \\
Seoul National University\\
Seoul 151-747\\
Republic of Korea} {zandc@snu.ac.kr}
\address{
Department of Mathematical Sciences \\
KAIST \\
Daejeon 305-701 \\
Republic of Korea} {jkkoo@math.kaist.ac.kr}
\address{% Corresponding Author
Department of Mathematics\\
Hankuk University of Foreign Studies\\
Yongin-si, Gyeonggi-do 449-791\\
Republic of Korea } {dhshin@hufs.ac.kr}

\end{document}